\documentclass[preprint,12pt]{elsarticle}
\usepackage[utf8]{inputenc}
\usepackage{lmodern,microtype,amsmath,amssymb,amsthm,mathtools,booktabs,array}
\usepackage[hidelinks]{hyperref}
\biboptions{numbers,sort&compress}
\journal{Journal of Number Theory}
\hypersetup{pdftitle={Bounded continued fractions and quadratic orders in Q(sqrt(127))},
 pdfsubject={Research manuscript prepared for the Journal of Number Theory},
 pdfkeywords={continued fractions, quadratic orders, bounded partial quotients}}
\newtheorem{theorem}{Theorem}[section]

\newtheorem{lemma}[theorem]{Lemma}
\newtheorem{corollary}[theorem]{Corollary}
\newtheorem{conjecture}[theorem]{Conjecture}
\theoremstyle{definition}

\newtheorem{example}[theorem]{Example}
\theoremstyle{remark}
\newtheorem{remark}[theorem]{Remark}
\newcommand{\Z}{\mathbb Z}
\newcommand{\Q}{\mathbb Q}
\newcommand{\F}{\mathbb F}
\newcommand{\tr}{\operatorname{tr}}
\newcommand{\rev}{\operatorname{rev}}
\newcommand{\per}[1]{[\overline{#1}]}
\newcommand{\OO}{\mathcal O}
\newcommand{\BB}{\mathcal B}
\newcommand{\eps}{\varepsilon}
\newcommand{\mat}[4]{\begin{pmatrix}#1&#2\\#3&#4\end{pmatrix}}
\numberwithin{equation}{section}

\begin{document}
\nocite{BGS10,BGS11,BK14,BK17,BK18,BV12,Buell89,Duke88}
\nocite{ELMV09,FineWilf65,FK14,GHR25,Good41,Helfgott08,Hensley89,Hensley92}
\nocite{Hensley96,Huang15,JP01,JP18,Kan17,Kont13,Kont16,Kont22}
\nocite{KL26,Lothaire02,MOW19,Malagoli17,McM09,McM12,Mercat13,MMS26}
\nocite{Niederreiter86,RS25,RME25,Sarnak82,Series85,Shk26,Shulga23,WD24}
\nocite{Wilson80,Woods78,Zaremba72,Zhang26}

\begin{frontmatter}
\title{Bounded continued fractions and quadratic orders in $\Q(\sqrt{127})$}
\author[ajman]{Shirali Kadyrov\corref{cor1}}
\ead{s.kadyrov@ajman.ac.ae}

\author[sdu]{Aiken Kazin}
\ead{aiken.kazin@sdu.edu.kz}

\author[sdu]{Alibek Orynbassar}
\ead{alibek.orynbassar@sdu.edu.kz}

\author[asu]{Nurdaulet Shynarbek}
\ead{n.shynarbek@innotech.edu.kz}

\cortext[cor1]{Corresponding author}

\affiliation[ajman]{
    organization={Ajman University},
    city={Ajman},
    country={United Arab Emirates}
}

\affiliation[sdu]{
    organization={SDU University},
    city={Kaskelen},
    country={Kazakhstan}
}

\affiliation[asu]{
    organization={Innotech powered by Arizona State University},
    city={Almaty},
    country={Kazakhstan}
}
\begin{abstract}
We study the least periods and multiplier orders in two constructions of
Mercat for continued fractions with bounded partial quotients. Our main
explicit result concerns a family in $\Q(\sqrt{127})$ with partial quotients
at most $4$. For every integer $j\geq1$, we prove that its primitive period
has length $92j+2$ and its multiplier order has conductor $y_{4j}/4$, where
$(4730624+419775\sqrt{127})^k=x_k+y_k\sqrt{127}$. The totally positive unit
group of this order has index $4j$ in that of the maximal order. These exact
formulas give divisor-sum lower bounds at fixed traces and a lower bound of
order $(\log X)^2$ for the number of bounded continued fractions with
discriminant at most $X$. We also prove a primitive-period refinement of
Mercat's palindromic construction. Combining this refinement with Hensley's
counting theorem gives, for every $m\geq2$ and $\eta>0$, at least
$X^{\delta_m/2-\eta}$ quadratic fields, up to a positive constant, that admit
infinitely many purely periodic continued fractions with partial quotients
at most $m$. Here $\delta_m$ is the Hausdorff dimension of the corresponding
continued fraction Cantor set.
\end{abstract}
\end{frontmatter}

\section{Introduction}\label{sec:intro}
A real irrational number has an eventually periodic regular continued
fraction if and only if it is quadratic. We write
$\per{a_1,\ldots,a_r}$ for a purely periodic continued fraction. For a real
quadratic field $K$ and an integer $m\geq2$, let
\[
 \BB_m(K)=\{\per{w}\in K:\ w\text{ is a nonempty word on }\{1,\ldots,m\}\}.
\]
Elements of this set are counted as real numbers. Repeating a period does
not give a new element. McMullen asked the following question
\cite{McM09,McM12,Mercat13}.
\begin{conjecture}\label{conj:binary}
For every real quadratic field $K$, the set $\BB_2(K)$ is infinite.
\end{conjecture}
McMullen proved that a bound depending on $K$ is sufficient. Mercat
\cite{Mercat13} gave explicit constructions and proved that
$\BB_2(K)$ is infinite for infinitely many fields. A weaker uniform question
asks for one absolute $m$ that works for every real quadratic field.

In this paper we study the least periods and quadratic orders in two of
Mercat's constructions. Our main explicit result concerns
$K=\Q(\sqrt{127})$ and partial quotients at most $4$. We determine the
conductors and the indices of the totally positive unit groups for every
member of the family. This gives lower bounds for the number of bounded
continued fractions at fixed traces and below a discriminant bound. We also
give a quantitative count of fields in which Mercat's palindromic construction
proves infinitude. These results do not settle Conjecture~\ref{conj:binary}
for an arbitrary prescribed field.

\subsection{An explicit family in \texorpdfstring{$\Q(\sqrt{127})$}{Q(sqrt(127))}}
The choice of $127$ has a historical reason. Mercat
\cite[Remarks 1.7 and 8.4]{Mercat13} reports a computation giving a purely
periodic continued fraction with partial quotients $1,2$ for every
squarefree $1<d<127$. He also gives a binary period of length $78$ in
$\Q(\sqrt{19})$. Thus $127$ is the next squarefree radicand after the range
covered by that computation.

The field $\Q(\sqrt{127})$ has maximal order $\Z[\sqrt{127}]$, discriminant
$508$, and fundamental unit
\begin{equation}\label{eq:epsilon127}
 \eps=4730624+419775\sqrt{127},\qquad N(\eps)=1.
\end{equation}
This follows from the standard continued fraction calculation
\[
 \sqrt{127}=[11;\overline{3,1,2,2,7,11,7,2,2,1,3,22}].
\]
Write $\eps^k=x_k+y_k\sqrt{127}$. The field has no unit of norm $-1$,
since $x^2-127y^2=-1$ is impossible modulo $127$.

We recall two standard terms used in the statements. A finite word is
\emph{primitive} if it is not a proper power of a shorter word. For a
quadratic irrational $\alpha\in K$, its multiplier order is
\[
 \OO(\alpha)=\{\beta\in K:\beta(\Z+\Z\alpha)\subseteq\Z+\Z\alpha\}.
\]
We write $\OO_f=\Z+f\OO_K$ for the order of conductor $f$, and
$\Delta(\alpha)$ for the discriminant of the primitive quadratic polynomial
of $\alpha$. The group of totally positive units of an order $\OO$ is denoted
by $\OO^{\times,+}$.

Define the following $45$-digit palindrome:
\begin{equation}\label{eq:Pword}
 P=\mathtt{213112111111131221312242213122131111111211312}.
\end{equation}
It starts and ends in $2$ and contains exactly one $4$. Write $P=2M2$, so
$M$ is a palindrome of length $43$, and set
\begin{equation}\label{eq:127words}
 \begin{gathered}
 a=P3,\qquad b=1131,\qquad c=2M31M2,\\
 z_j=b\,a^{j-1}c\,\rev(a)^{j-1}\quad(j\geq1).
 \end{gathered}
\end{equation}
This is Mercat's Proposition 7.1 with parameters $i=3$, $j=2$ there; the
index $j$ in \eqref{eq:127words} parametrizes our family.

\begin{theorem}\label{thm:127}
Let $K=\Q(\sqrt{127})$. For every integer $j\geq1$:
\begin{enumerate}
\item $\alpha_j=\per{z_j}$ belongs to $\Q(\sqrt{127})$ and uses only
the digits $1,2,3,4$;
\item its primitive period has length $92j+2$;
\item its multiplier order is $\OO_{f_j}$, where
\begin{equation}\label{eq:fj}
 f_j=\frac{y_{4j}}4
     =\frac{\eps^{4j}-\eps^{-4j}}{8\sqrt{127}},
 \qquad
 \Delta(\alpha_j)=\frac{(\eps^{4j}-\eps^{-4j})^2}{16};
\end{equation}
\item the index $[\OO_K^{\times,+}:\OO_{f_j}^{\times,+}]$ is $4j$, and
\begin{equation}\label{eq:indexfour}
 [\OO_{f_j}:\Z[\eps^{4j}]]=4.
\end{equation}
\end{enumerate}
In particular the family consists of infinitely many distinct periods, up to cyclic rotation, and supplies infinitely many elements of $\BB_4(\Q(\sqrt{127}))$.
\end{theorem}
The word construction is Mercat's Proposition 7.1. The additional arithmetic
step is an exact greatest common divisor calculation for its quadratic
coefficients. This calculation determines the order and also proves that
the displayed period is the least period. The proof is given in
Section~\ref{sec:127}.

\subsection{Counting bounded continued fractions}\label{subsec:counts}
For a real quadratic field $K$, let $\eps_K>1$ generate
$\OO_K^{\times,+}$ and put $T_k=\eps_K^k+\eps_K^{-k}$.
For an order $\OO_f$, write
\[
 k(f)=[\OO_K^{\times,+}:\OO_f^{\times,+}].
\]
Let $A_{K,m}(k)$ count the elements of $\BB_m(K)$ whose multiplier orders
have $k(f)=k$. Let $R_{K,m}(T_k)$ count those whose multiplier orders contain
$\eps_K^k$. In the case $K=\Q(\sqrt{127})$ we use the subscript $127$.
Also let $N_{127,4}(X)$ count the elements of
$\BB_4(\Q(\sqrt{127}))$ with $\Delta(\alpha)\leq X$.
Write $\sigma_1(r)=\sum_{d\mid r}d$ and $\tau(r)=\sum_{d\mid r}1$.
\begin{corollary}\label{cor:counts127}
For every $r\geq1$,
\begin{equation}\label{eq:127counts}
 A_{127,4}(4r)\geq92r+2,\qquad
 R_{127,4}(T_{4r})\geq92\sigma_1(r)+2\tau(r).
\end{equation}
If
\begin{equation}\label{eq:JX}
 J(X)=\left\lfloor\frac{\operatorname{arsinh}(2\sqrt X)}{4\log\eps}\right\rfloor,
\end{equation}
then
\begin{equation}\label{eq:cumulative}
 N_{127,4}(X)\geq46J(X)^2+48J(X).
\end{equation}
In particular,
\begin{equation}\label{eq:liminf}
 \liminf_{X\longrightarrow\infty}
 \frac{N_{127,4}(X)}{(\log X)^2}
 \geq\frac{23}{32(\log\eps)^2}.
\end{equation}
\end{corollary}
The proof uses the distinct cyclic rotations of each primitive period.
The conductors are strictly increasing, so different parameters give
different sets of roots. These bounds are proved at the end of
Section~\ref{sec:127}.

\subsection{Counting quadratic fields}
Let
\[
 E_m=\{[0;a_1,a_2,\ldots]:1\leq a_i\leq m\},\qquad
 \delta_m=\dim_H(E_m).
\]
Define
\[
 G_m(X)=\#\left\{\begin{array}{ll}
 1<d\leq X:\!&d\text{ squarefree},\quad d\equiv3\pmod4,\\
 &|\BB_m(\Q(\sqrt d))|=\infty
 \end{array}\right\}.
\]
\begin{theorem}\label{thm:fields}
For every fixed $m\geq2$ and $\eta>0$,
\begin{equation}\label{eq:fieldcount}
 G_m(X)\gg_{m,\eta}X^{\delta_m/2-\eta}
 \qquad(X\longrightarrow\infty).
\end{equation}
For the binary alphabet the exponent $\delta_2/2$ is
$0.265640253138602570812\ldots$.
\end{theorem}
The qualitative infinitude of these fields follows from Mercat's
construction. For the quantitative bound, we combine Hensley's matrix count
with estimates for the number of constructions having a given trace or a
given quadratic field. The proof is given in Section~\ref{sec:fields}.

\subsection{Related work}
Woods and Wilson \cite{Woods78,Wilson80} obtained earlier results in fixed
quadratic fields through approximation spectra. McMullen \cite{McM09}
developed the dynamical and arithmetic form of the problem. His stronger
arithmetic-chaos question \cite{McM12} concerns the abundance of bounded
periods in each field. Mercat's explicit constructions \cite{Mercat13}
are the starting point of this paper. Malagoli \cite[\S3.2]{Malagoli17}
reviews the real-number questions and develops analogues in function fields.

The relation between continued fractions and the modular geodesic flow is
described by Series \cite{Series85}. Quadratic forms and their unit groups
provide the arithmetic description; see Buell \cite{Buell89} and Sarnak
\cite{Sarnak82}. Duke \cite{Duke88} and
Einsiedler--Lindenstrauss--Michel--Venkatesh \cite{ELMV09} study the
distribution of arithmetic geodesics and periodic torus orbits.
Bourgain--Kontorovich \cite{BK17} prove the existence of infinitely many
low-lying fundamental geodesics. In our explicit family, the field is fixed
and the order discriminants are $508f_j^2$. Recent work of
Rieutord--Moreira--Erazo \cite{RME25} also uses periodic words to study the
Markov and Lagrange spectra. Wang--Deng \cite{WD24} study continued
fractions whose partial quotients lie in a quadratic ring; in the present
paper the partial quotients are positive rational integers.

Kontorovich's surveys \cite{Kont13,Kont16} explain the methods for thin
semigroups. His work on trace sets \cite{Kont22}, including an argument of
Bourgain, and Bourgain--Kontorovich \cite{BK18} give quantitative results
for restricted continued fraction semigroups. A large set of traces over
many fields does not by itself give a trace from the unit sequence of one
prescribed field. Here these traces are obtained from explicit identities.

The dimension theory of restricted continued fractions began with Good
\cite{Good41}. Hensley's counting and dimension results
\cite{Hensley89,Hensley92,Hensley96} are used in many quantitative arguments.
Jenkinson--Pollicott \cite{JP01,JP18} give effective dimension methods,
including the rigorous value
\[
 \delta_2=0.531280506277205141624468647\ldots.
\]
Our exponent in Theorem~\ref{thm:fields} is stated using the exact dimension.

Zaremba's conjecture \cite{Zaremba72} asks for a rational continued fraction
with uniformly bounded partial quotients for every denominator.
Niederreiter \cite{Niederreiter86} treated dyadic denominators explicitly.
Bourgain--Kontorovich \cite{BK14} proved the density-one statement with
bound $50$, and Huang \cite{Huang15} reduced the bound to $5$.
Gonz\'alez Robert--Hussain--Shulga \cite{GHR25} study a related
Zaremba question for Hurwitz continued fractions over the Gaussian integers.
Frolenkov--Kan \cite{FK14} and Kan \cite{Kan17} improved
positive-proportion results for general alphabets. Shulga \cite{Shulga23}
gives a bound involving the radical of the denominator. The relevant
expansion and congruence-counting results include Helfgott \cite{Helfgott08},
Bourgain--Gamburd--Sarnak \cite{BGS10,BGS11}, Bourgain--Varj\'u
\cite{BV12}, and Magee--Oh--Winter \cite{MOW19}, with its appendix by
Bourgain, Kontorovich, and Magee.

There are also several recent results and claims. The published 2026 paper
of Moshchevitin--Murphy--Shkredov \cite{MMS26} improves Korobov's bound,
in particular to $O(\log p/\log\log p)$ for prime denominators.
Shkredov's preprint \cite[Corollary 1]{Shk26}, revised in June 2026,
announces an absolute bound for prime and squarefree denominators.
Zhang's preprint \cite{Zhang26}, revised in May 2026, announces
Zaremba's conjecture with a large absolute bound. These statements concern
an unspecified uniform bound. Mercat's implication, also recorded in
\cite[Theorem 3.2.3]{Malagoli17}, takes a Zaremba bound $z$ to a uniform
McMullen bound $z+1$. It does not give the binary bound in
Conjecture~\ref{conj:binary}. The exact formulas proved here do not depend on
these preprint claims.

Kopp--Lagarias \cite{KL26}, in a preprint revised in April 2026, study
class numbers of orders generated by units. The orders in
Theorem~\ref{thm:127} contain $\Z[\eps^{4j}]$ with index $4$.
Rickards--Stange \cite{RS25} prove reciprocity obstructions for certain
semigroup orbits. This concerns the limits of local-global arguments and
does not decide the binary continued fraction question.

\subsection{Plan of the paper}
Section~\ref{sec:prelim} recalls the relation between continued fraction
matrices, multiplier orders, and units. Section~\ref{sec:127} proves
Theorem~\ref{thm:127} and Corollary~\ref{cor:counts127}.
Section~\ref{sec:pal} proves a refinement concerning primitive periods of Mercat's
palindromic construction, using the Fine--Wilf theorem
\cite{FineWilf65,Lothaire02}. Section~\ref{sec:fields} proves
Theorem~\ref{thm:fields}.

\section{Continued fraction matrices and quadratic orders}\label{sec:prelim}
For a digit $a\geq1$, and a finite word $w=a_1\cdots a_r$, put
\[
 B_a=\mat a1 10,\qquad P_w=B_{a_1}\cdots B_{a_r},\qquad P_{\varnothing}=I.
\]
Concatenation is written without commas, $w^n$ means concatenation of $n$
copies, and $\rev(w)$ reverses the word. The monoid generated by the matrices $B_a$ is free:
the matrix determines its complete word, including terminal digits $1$.
This follows from the Euclidean factorization of positive unimodular matrices
\cite[\S2]{Mercat13}; using both columns removes the familiar ambiguity of a
finite scalar continued fraction. Transposition gives
$P_{\rev(w)}=P_w^t$.

If $P_w=\mat ab ce$, its positive fixed point satisfies
\begin{equation}\label{eq:fixedquad}
 c\alpha^2+(e-a)\alpha-b=0,
 \qquad \alpha=\per w,
 \qquad \Q(\alpha)=\Q\bigl(\sqrt{\tr(P_w)^2-4\det(P_w)}\bigr).
\end{equation}
Every such root is reduced: $\alpha>1$ and $-1<\alpha'<0$.
For $M=\mat ab ce\in\mathrm{SL}_2(\Z)$, define
\begin{equation}\label{eq:content}
 g(M)=\gcd(b,c,a-e).
\end{equation}
The primitive quadratic polynomial of its reduced fixed point has discriminant
\begin{equation}\label{eq:primitiveD}
 \Delta(\alpha)=\frac{\tr(M)^2-4}{g(M)^2}.
\end{equation}

Let $K$ have fundamental discriminant $D_K$ and let
$\OO_f=\Z+f\OO_K$ be its order of conductor $f$. With $\eps_K$ as in Section~\ref{subsec:counts}, write
\[
 T_k=\eps_K^k+\eps_K^{-k},\qquad
 \eps_K^k=\frac{T_k+b_k\sqrt{D_K}}2.
\]
Then $T_k,b_k$ are integers, $b_k>0$ for $k\geq1$, and
\begin{equation}\label{eq:firstindex}
 k(f)=\min\{k\geq1:f\mid b_k\},\qquad
 f\mid b_k\ \Longleftrightarrow\ k(f)\mid k.
\end{equation}
For completeness, $k(f)$ is the order of the image of $\eps_K$ in the finite
group
$ (\OO_K/f\OO_K)^\times/(\Z/f\Z)^\times$.

\begin{lemma}[Multiplier order]\label{lem:order}
Let $A_0X^2+B_0X+C_0$ be a primitive polynomial for a reduced quadratic
irrational $\alpha$, with $A_0>0$. The multiplier ring of
$\Z+\Z\alpha$ is
\[
 \Z+A_0\alpha\Z
   =\Z+\frac{-B_0+\sqrt{\Delta(\alpha)}}2\Z.
\]
In particular it is $\OO_f$ if $\Delta(\alpha)=D_Kf^2$.
\end{lemma}
\begin{proof}
A multiplier belongs to $\Z+\Z\alpha$, since it acts on $1$, so write it
$r+s\alpha$ with $r,s\in\Z$. Multiplication by $\alpha$ preserves the lattice
precisely when $A_0\mid sB_0$ and $A_0\mid sC_0$. Primitivity implies
$A_0\mid s$, prime by prime. The displayed formula and the order discriminant
follow. This is the usual form--order correspondence \cite{Buell89}.
\end{proof}

The counting functions in Section~\ref{subsec:counts} are finite. Equivalently, a root counted by $R_{K,m}(T_k)$ is
fixed by a product of continued fraction matrices of determinant one and trace $T_k$.
This follows from the usual inequalities for a reduced quadratic irrational. If the multiplier is
$\eps_K^k=(t+u\sqrt\Delta)/2$, where $u=b_k/f$, its matrix is
\begin{equation}\label{eq:unitmatrix}
 M=\mat{(t-B_0u)/2}{-C_0u}{A_0u}{(t+B_0u)/2}.
\end{equation}
The reduced-root inequalities and integrality give
\[
 a>\max(b,c),\qquad \min(b,c)\geq e\geq1.
\]
For example,
$c-e=((2A_0-B_0)u-t)/2>-\eps_K^{-k}>-1$, so $c\geq e$;
apply the same argument to $-1/\alpha'$ for the other inequality.
The continued fraction matrix factorization \cite[\S2]{Mercat13} supplies the word.
If $\alpha$ has bounded periodic digits, that factorization follows its period
and hence has the same digit bound. At a fixed trace the displayed cone contains
only finitely many integer matrices.
Equation~\eqref{eq:firstindex} now gives the useful exact identity
\begin{equation}\label{eq:divisoridentity}
 R_{K,m}(T_k)=\sum_{h\mid k}A_{K,m}(h).
\end{equation}

\begin{lemma}[Primitivity and unit group index]\label{lem:primitiveindex}
Suppose a word of even length with positive integer digits $w$ is primitive and $P_w$ has expanding
eigenvalue $\eps_K^k$. Its fixed point has unit group index $k$.
Conversely, if $K$ has no units of norm $-1$ and the fixed point has unit group index $k$, then any word matrix with expanding eigenvalue $\eps_K^k$ is primitive.
\end{lemma}
\begin{proof}
For the first assertion let $h=k(f)$. Then $h\mid k$, and
\eqref{eq:unitmatrix} gives a continued fraction matrix $V$ for $\eps_K^h$.
Both matrices have the same two eigenlines, so $P_w=V^{k/h}$.
Freeness of this monoid implies that $w$ is a word power if $h<k$.
For the converse, a proper word root would yield an algebraic-integer unit
in the multiplier order with smaller expanding eigenvalue. Its norm cannot
be $-1$ under the stated hypothesis, so it would have an index strictly
between $0$ and $k$, a contradiction.
\end{proof}

\section{Proofs for the family in \texorpdfstring{$\Q(\sqrt{127})$}{Q(sqrt(127))}}\label{sec:127}
\subsection{The trace identity}
Put $A=P_a$, $B=P_b$, $C=P_c$ and use column vectors. Exact multiplication gives
\begin{equation}\label{eq:A127}
 A=\mat{80862222221311}{24069743575040}{29069695697535}{8652991496191}.
\end{equation}
The remaining determinant and trace data are
\[
 B=\begin{pmatrix}9&7\\5&4\end{pmatrix},\qquad
 \det A=\det B=\det C=1.
\]
\[
 t=\tr A=89515213717502=9461248^2-2.
\]
Set
\begin{equation}\label{eq:uv127}
 J=\mat0{-1}10,\qquad v=\binom32,\qquad
 u=\binom{24069743575040}{8652991496191}=Ae_2.
\end{equation}
The two exact rank-one identities are
\begin{equation}\label{eq:rank127}
 B=vv^t-J,\qquad C=uu^t-J.
\end{equation}
For $n=j-1$, let $w=A^nu$. Since $A^nJ(A^t)^n=J$, the family matrix is
\begin{equation}\label{eq:W127}
 W_n=P_{z_j}=BA^nC(A^t)^n=(vv^t-J)(ww^t-J).
\end{equation}
Traces of symmetric--skew products vanish, and $J^2=-I$, so
\begin{equation}\label{eq:trace127first}
 \tr W_n=(v^tw)^2-2.
\end{equation}
The relation $3A_{12}+2A_{22}=\tr A$ says
$v^tAe_2=\tr A$; also $v^te_2=2=\tr I$.
Cayley--Hamilton expresses every power of $A$ in the span of $I,A$.
Consequently
\[
 v^tw=v^tA^{n+1}e_2=\tr A^{n+1}.
\]
It follows from \eqref{eq:trace127first} that
\begin{equation}\label{eq:trace127}
 \det W_n=1,\qquad \tr W_n=\tr A^{2n+2}=\eps^{4j}+\eps^{-4j}.
\end{equation}
In the last equality the expanding eigenvalue of $A$ is $\eps^2$, by its
determinant and trace. Thus the field assertion is already proved.

\subsection{The common divisor of the quadratic coefficients}\label{subsec:content}
We now compute the greatest common divisor of the quadratic coefficients.
\begin{lemma}\label{lem:contentfour}
For every $n\geq0$, the matrix in \eqref{eq:W127} satisfies $g(W_n)=4$.
\end{lemma}
\begin{proof}
Write $w=(h,k)^t$. Directly from \eqref{eq:W127},
\begin{align}\label{eq:Wentries}
 (W_n)_{12}&=9hk+7k^2+9,\notag\\
 (W_n)_{21}&=5h^2+4hk-4,\\
 (W_n)_{11}-(W_n)_{22}&=9h^2+2hk-4k^2-12.\notag
\end{align}
The residues
\[
 A\equiv\mat3033\pmod4,\qquad u\equiv\binom03\pmod4
\]
show that $4\mid h$ and $k$ is odd for every $n$.
All three expressions in \eqref{eq:Wentries} are divisible by $4$, while
$(W_n)_{21}\equiv-4\pmod{16}$. Thus the $2$-part of $g(W_n)$ is exactly $4$.

Suppose an odd prime $p$ divides $g(W_n)$. Then $W_n$ is scalar modulo $p$,
say $sI$. Its product description gives
$ww^t-J=sB^{-1}$. Comparing the differences of the two off-diagonal entries gives $s=-1$.
With $z=(2,-3)^t$ one has $B^{-1}=zz^t+J$, whence
\begin{equation}\label{eq:oddcontent}
 ww^t=-zz^t\pmod p.
\end{equation}
In particular $k^2=-9\pmod p$ and, since $v^tz=0$,
$v^tw=0\pmod p$.

An invariant integral form restricts the possible primes. Put
\begin{equation}\label{eq:Qform}
 Q_0=\frac{JA-A^tJ}{4930}
  =\mat{-11792979999}{14646902784}{14646902784}{9764601856}.
\end{equation}
One has $A^tQ_0A=Q_0$, either by direct multiplication or by
$A^tJA=J$. Therefore $w^tQ_0w=u^tQ_0u$.
Taking the trace after multiplying \eqref{eq:oddcontent} by $Q_0$ shows that
$p$ divides the following fixed integer:
\begin{align}\label{eq:Hfactor}
 H&=u^tQ_0u+z^tQ_0z=-125288734844\notag\\
  &=-2^2\cdot7^2\cdot13\cdot43\cdot557\cdot2053.
\end{align}
The primes $7$ and $43$ are excluded by $k^2=-9$, since they are $3$ modulo $4$.
For the three remaining primes, direct exponentiation gives:
\begin{equation}\label{eq:oddorders}
 \begin{aligned}
 A^3&=-I\pmod{13},\qquad A^{139}=-I\pmod{557},\\
 A^{513}&=-I\pmod{2053}.
 \end{aligned}
\end{equation}
Thus the order of $A$ modulo each such prime divides twice an odd integer.
But $v^tw=\tr A^{n+1}=0\pmod p$ would imply
$(A^{n+1})^2=-I$ by Cayley--Hamilton. The element $A^{n+1}$ would have order
$4$, which is impossible. No odd prime divides $g(W_n)$.
Each equality in \eqref{eq:oddorders} can be checked by repeated squaring.
For convenience, the residues of $A$ are listed in Table~\ref{tab:modular}.
\end{proof}

\begin{table}[ht]
\centering
\caption{Congruences for the remaining possible odd prime divisors of $g(W_n)$.}
\label{tab:modular}
\begin{tabular}{@{}rccc@{}}
\toprule
$p$ & $A\bmod p$ & odd $h$ & $A^h\bmod p$\\
\midrule
13 & $\begin{psmallmatrix}11&7\\12&3\end{psmallmatrix}$ & 3 & $-I$\\[3pt]
557 & $\begin{psmallmatrix}165&148\\261&278\end{psmallmatrix}$ & 139 & $-I$\\[3pt]
2053 & $\begin{psmallmatrix}521&517\\1019&1023\end{psmallmatrix}$ & 513 & $-I$\\
\bottomrule
\end{tabular}
\end{table}

\subsection{The order and the least period}
\begin{proof}[Proof of Theorem~\ref{thm:127}]
The words have lengths $|a|=46$, $|b|=4$, $|c|=90$, so $|z_j|=92j+2$.
By \eqref{eq:trace127}, the trace discriminant is
$508y_{4j}^2$. Lemma~\ref{lem:contentfour}, \eqref{eq:primitiveD}, and
Lemma~\ref{lem:order} give conductor $f_j=y_{4j}/4$ and
\eqref{eq:fj}. Integrality also follows from $4\mid y_4$ and the formula for
powers of $\eps^4$.

The coefficient sequence satisfies
\[
 y_0=0,\quad y_1=419775,\quad
 y_{r+1}=9461248\,y_r-y_{r-1}.
\]
It is strictly increasing, and for $r\geq2$,
\begin{equation}\label{eq:ygap}
 y_r>(9461248-1)y_{r-1}>4y_{r-1}.
\end{equation}
Hence $f_j=y_{4j}/4>y_{4j-1}$. No positive $y_r$ with $r<4j$ can be a
multiple of $f_j$, whereas $f_j\mid y_{4j}$. The unit group index is therefore
exactly $4j$. Since the field has no units of norm $-1$,
Lemma~\ref{lem:primitiveindex} proves that $z_j$ is primitive. This gives the
claimed least period, not merely a displayed period.

Finally $\Z[\eps^{4j}]=\Z+y_{4j}\sqrt{127}\Z=\OO_{y_{4j}}$.
The ratio of conductors gives \eqref{eq:indexfour}. The conductors increase
strictly with $j$, so the cycles are distinct.
\end{proof}

\begin{remark}
Primitivity can also be checked combinatorially. The word has $2j$ occurrences
of $4$ and length $92j+2$, whose greatest common divisor is $2$.
Only a square could be a proper word power. For $j\geq3$ the overlap of its
two long blocks would force $\rev(a)$ to be the left cyclic shift of $a$ by
$|a|-3$. Palindromicity of $P$ makes it the shift by $|a|-1$, forcing period
$2$ in $a$, which contains four distinct digits. For $j=1,2$, the digit at
zero-based index $1$ differs from that at the corresponding index $48,94$
in the other half. This is an independent elementary check.
\end{remark}

\begin{example}\label{ex:first}
For $j=1$, the word $z_1=bc$ has primitive length $94$. Its matrix has trace
\[
 T_4=8012973486890058230661120002,
\]
and its multiplier order has conductor
\[
 f_1=88879552292132849341689600.
\]
Thus this single primitive period gives $94$ distinct reduced quadratic
irrationals of discriminant $508f_1^2$ in the same field.
\end{example}

\subsection{Proof of the counting bounds}
\begin{proof}[Proof of Corollary~\ref{cor:counts127}]
Each primitive word $z_j$ has $92j+2$ distinct cyclic rotations. Their
continued fractions are distinct by uniqueness of infinite irrational
continued fractions. Rotation conjugates the matrix by a unimodular digit
matrix, so it preserves $g(M)$ and the trace, or equivalently preserves the
multiplier order. The first inequality follows. Formula
\eqref{eq:divisoridentity}, summed over $j\mid r$, gives the second.

The discriminant in \eqref{eq:fj} is
$\tfrac14\sinh^2(4j\log\eps)$, so $j\leq J(X)$ is exactly the condition
that it be at most $X$. Different $j$ give different orders and therefore
disjoint sets of roots. Summing $92j+2$ for $1\leq j\leq J(X)$ gives
\eqref{eq:cumulative}. Finally
$J(X)=\log X/(8\log\eps)+O(1)$, proving \eqref{eq:liminf}.
\end{proof}
The bounds count distinct reduced quadratic irrationals. The first inequality
counts those whose orders have unit group index $4r$. The second also counts
the orders whose indices divide $4r$.

\section{Primitive periods from palindromic words}\label{sec:pal}
We next give a general quantitative refinement of Mercat's
Proposition 7.7 \cite{Mercat13}. This section is also the constructive input
for counting fields.
\begin{theorem}\label{thm:pal}
Let $m\geq2$ and let $s$ be any palindrome on $\{1,\ldots,m\}$, including the
empty word. Set
\begin{equation}\label{eq:generalwords}
 a=s11211,\quad b=2111,\quad c=s121111s,\quad
 w_n=b\,a^nc\,\rev(a)^n,\quad L=|s|+5.
\end{equation}
Let $A=P_a$ and let $\lambda>1$ be its expanding eigenvalue. Put
\[
 K=\Q\bigl(\sqrt{\tr(A)^2-4\det A}\bigr).
\]
Then every $\per{w_n}$ belongs to
$\BB_m(K)$. For every $n\geq5$, the word $w_n$ is primitive of length
$2(n+1)L$. If $\lambda^2=\eps_K^q$, then $q\geq1$ is an integer and
\begin{equation}\label{eq:generalcounts}
 A_{K,m}(qj)\geq2jL\quad(j\geq6),\qquad
 R_{K,m}(T_{qr})\geq2L\sum_{\substack{j\mid r\\j\geq6}}j\quad(r\geq1).
\end{equation}
\end{theorem}
\begin{proof}
Write $S=P_s=S^t$ and $\sigma=\det S\in\{1,-1\}$.
The fixed matrices are
\[
 R=P_{11211}=\mat{12}774,\quad
 B=P_{2111}=\mat8532,\quad
 D=P_{121111}=\mat{18}{11}{13}8.
\]
Thus $A=SR$, $C=SDS$, and $P_{w_n}=BA^nC(A^t)^n$.
With $v=(2,1)^t$, $u=(3,2)^t$ and $J$ as above,
\[
 B=2vv^t-J,\quad D=2uu^t+J,\quad R=2uv^t-J.
\]
For any matrix $X$, symmetric--skew cancellation gives
\begin{equation}\label{eq:generaltrace1}
 \tr(BXSDSX^t)=4(v^tXSu)^2+2\sigma\det X.
\end{equation}
Indeed $SJS=\sigma J$ and $XJX^t=\det(X)J$.
The matrix $A^nS$ is symmetric, since it is the alternating product
$SRS\cdots RS$. Consequently
$\tr A^{n+1}=\tr(A^nSR)=2v^tA^nSu$.
Since $\det A=-\sigma$, \eqref{eq:generaltrace1} yields
\begin{equation}\label{eq:generaltrace2}
 \det P_{w_n}=1,\qquad
 \tr P_{w_n}=(\tr A^{n+1})^2-2\det(A^{n+1})=\tr A^{2n+2}.
\end{equation}
The expanding eigenvalue is $\lambda^{2n+2}$, so the quadratic field is $K$.
Also $\lambda^2$ is a totally positive algebraic-integer unit, establishing
$\lambda^2=\eps_K^q$ for an integer $q\geq1$.

We prove primitivity explicitly. Let $p$ be the least period of the infinite
repetition of $w_n$; it divides $N=2(n+1)L$.
If $p\leq(n-1)L$, the block $a^n$ has periods $p,L$ and length at least
$p+L$. The Fine--Wilf theorem \cite{FineWilf65}, in the standard word
form \cite{Lothaire02}, gives period $\gcd(p,L)$. Since this divides $L$,
the infinite repetition of $a$ also has period $p$.
Shift each position of the initial block $b$ forward by an appropriate
multiple of $p$ into $a^n$. Periodicity forces $b$ to be the four digits
immediately preceding $a$ in its periodic repetition. But the last four digits
of $a$ are $1211$, whereas $b=2111$. Therefore $p>(n-1)L$.

For $n\geq5$, a proper repetition of exponent at least three would give
$p\leq N/3\leq(n-1)L$. Only a square remains.
If $p=N/2=(n+1)L$, shift the first long block by this half-length. Its
overlap with the reversed long block has length $(n-1)L+4\geq L$.
It forces $\rev(a)$ to be the left cyclic shift of $a$ by $L-4$.
But $s$ and $11211$ are palindromes, so $\rev(a)=11211s$ is its shift
by $L-5$. The shifts differ by one, making $a$ constant, contrary to the
digits $1,2$ in its terminal block. This proves primitivity.

The $2jL$ cyclic rotations give distinct roots. Their expanding eigenvalue is
$\eps_K^{qj}$, and Lemma~\ref{lem:primitiveindex} gives unit group index $qj$.
Equation~\eqref{eq:divisoridentity} proves the second inequality.
\end{proof}

\section{Proof of the field-counting theorem}\label{sec:fields}
\begin{proof}[Proof of Theorem~\ref{thm:fields}]
Use the sum of matrix entries as norm. Hensley's theorem
\cite{Hensley89}, also stated in \cite[Theorem 1.7]{Kont22}, gives
\begin{equation}\label{eq:hensleycount}
 \#\{P_v:\ v\text{ on }\{1,\ldots,m\},\ \|P_v\|\leq Y\}
 \asymp_m Y^{2\delta_m}.
\end{equation}
An equivalent matrix norm only changes constants. Passing between even word
lengths and all word lengths also affects only constants, by appending one
fixed digit when necessary.

The reductions of $B_1,B_2$ generate all six elements of
$\mathrm{GL}_2(\F_2)$, represented by
$\varnothing,1,11,2,12,21$. Append to each word a suffix from this list
so that its matrix $V$ satisfies
\[
 V\equiv L_0:=\mat1011\pmod2.
\]
The suffix has norm at most $7$, so $\|V\|\leq7Y$.
Any output has at most six inputs: for each chosen suffix, right inversion
recovers the input. We retain $\gg_m Y^{2\delta_m}$ distinct matrices $V$.
If $v$ is the word of $V$, put
\begin{equation}\label{eq:countmatrices}
 s=v1\rev(v),\qquad S=VB_1V^t,\qquad A=SR,
 \qquad R=P_{11211}.
\end{equation}
Freeness of the monoid makes the palindromes $s$, hence the matrices $S$,
distinct. They satisfy
$\det S=-1$ and $S\equiv I\pmod2$, since $L_0B_1L_0^t=I$.
The matrices $A$ therefore have determinant one.
Their traces $T$ are multiples of four, and
\begin{equation}\label{eq:tracebound}
 T\leq\|S\|\|R\|\leq(7Y)^2\cdot3\cdot30=4410Y^2.
\end{equation}
Thus there are $\gg_m N^{\delta_m}$ distinct matrices of trace at most $N$.

We bound collisions at a fixed trace. Write $S=\mat ab bc$;
then $ac-b^2=-1$ and $T=12a+14b+4c$. Set
\begin{equation}\label{eq:UV}
 U=T-3a-2b>0,\qquad V_1=T+4a+2b>0.
\end{equation}
Direct expansion gives
\[
 UV_1=T^2+4(ac-b^2)=T^2-4.
\]
Conversely a divisor $U$ of $T^2-4$, and hence $V_1=(T^2-4)/U$, determines
the entire matrix $S$:
\begin{equation}\label{eq:recover}
 a=U+V_1-2T,\quad b=\frac{7T-4U-3V_1}{2},\quad
 c=\frac{16U+9V_1-24T}{4}.
\end{equation}
The set of matrices $A$ of a given trace has size at most $\tau(T^2-4)$.
The elementary divisor bound gives, for every $\rho>0$, a uniform upper
bound $O_\rho(N^\rho)$ for these sets when $T\leq N$.
Hence these matrices give $\gg_{m,\rho}N^{\delta_m-\rho}$ distinct traces.

Since $4\mid T$, the number $(T/2)^2-1$ is odd and is $3$ modulo $4$.
Write it $d u^2$ with $d$ squarefree. Then $u$ is odd,
$d\equiv3\pmod4$, and $d<T^2/4$.
Within any one quadratic field, each expanding eigenvalue is a totally positive
unit $\eps_K^r$. Such a unit has integral trace at least $3$, so uniformly
\[
 \eps_K\geq(3+\sqrt5)/2.
\]
The condition $\tr(\eps_K^r)\leq N$ permits at most $O(\log N)$ indices,
uniformly in $K$. Passing from distinct traces to distinct fields loses only
this logarithmic factor.
Every field obtained in this construction satisfies the infinitude conclusion of
Theorem~\ref{thm:pal}. Taking $N$ of order $\sqrt X$ and absorbing
$N^\rho\log N$ into $X^\eta$ proves \eqref{eq:fieldcount}.
\end{proof}

\begin{remark}
The congruence restriction is substantive: a squarefree integer $3$ modulo
$4$ has a prime factor $3$ modulo $4$, so its field has no units of norm $-1$. The construction therefore supplies quantitative binary infinitude
even within this class. It applies unchanged to $m=3$ and every larger
fixed alphabet, but does not prescribe an individual radicand.
\end{remark}

\section{Further questions}
The orders in Theorem~\ref{thm:127} are determined exactly. A natural next
question is whether a comparable family in $\Q(\sqrt{127})$ can be obtained
with partial quotients at most $3$, or only $1,2$. For the general construction,
one may ask whether suitable palindromic words can be found in every
prescribed field for a small fixed bound. Theorem~\ref{thm:fields} counts
fields but does not answer this question for a chosen radicand.

\section*{Declaration of generative AI and AI-assisted technologies}
ChatGPT (OpenAI) GPT 6.1 Sol was used to assist in
arithmetic verification, organization, and language of the manuscript.
The arguments and the numerical identities used in the proofs are stated in
full in the article. Responsibility for the submitted work rests with the
human authors.

\bibliographystyle{elsarticle-num}
\bibliography{references}
\end{document}